\documentclass[reqno,12pt]{amsart}
\usepackage{a4wide}

\usepackage{amsmath,amssymb}
  \usepackage{paralist}
  \usepackage{graphics} 
  \usepackage{epsfig} 
\usepackage{graphicx}  \usepackage{epstopdf}
 \usepackage[colorlinks=true]{hyperref}

\newtheorem{theorem}{Theorem}[section]

\newtheorem{lemma}[theorem]{Lemma}

\theoremstyle{definition}

\newtheorem{remark}{Remark}

\newcommand{\R}{\mathbb{R}}

\newcommand{\Z}{\mathbb{Z}}

\renewcommand{\S}{\mathbb{S}}
\newcommand{\T}{\mathbb{T}}

\newcommand{\FE}{\mathbf{E}}

\newcommand{\FP}{\mathbf{P}}

\newcommand{\FI}{\mathbf{I}}

\newcommand{\na}{\nabla}

\newcommand{\ga}{\gamma}

\newcommand{\si}{\sigma}
\newcommand{\pa}{\partial}

\newcommand{\Ga}{\Gamma}

\renewcommand{\Re}{\operatorname{Re}}

\begin{document}

\title[Global Solutions of Non-cutoff Boltzmann Equation]{Global solutions in $W_k^{\zeta,p}L^\infty_TL^2_v$ for the Boltzmann equation without cutoff}

\author[H.-Y. Zhang]{Haoyu Zhang}
\address[H.-Y. Zhang]{Department of Mathematics, Sichuan University, Chengdu 610065, P.R.~China}
\email{hyzhang116@163.com}

\begin{abstract}
The Boltzmann equation without an angular cutoff in a three-dimensional periodic domain is considered. The global-in-time existence of solutions in a function space $ W_k^{\zeta,p}L^\infty_TL^2_v $ with $p>1$ and $\zeta>3(1-\frac{1}{p})$ is established in the perturbation framework and the long-time behavior of solutions is also obtained for both hard and soft potentials. The proof is based on several norm estimates. 
\end{abstract}

 
\date{\today}
\maketitle


\thispagestyle{empty}


\section{Introduction}
In this paper we are concerned with the Cauchy problem on the non-cutoff Boltzmann equation in a torus domain
\begin{equation}\label{Beq}
\pa_t F+v\cdot \na_x F=Q(F,F)
\end{equation}
with
\begin{equation}\label{Bid}
F(0,x,v)=F_0(x,v).
\end{equation} 
Here the unknown $F=F(t,x,v)\geq 0$ denotes the density distribution function of gas particles with position $x\in \T^3$ and velocity $v\in \R^3$ at time $t>0$. The right hand side of (\ref{Beq}) is the non-cutoff Boltzmann collision operator of the form
\begin{equation}\label{bBop}
Q(G,F)(v)=\int_{\R^3}\int_{\S^2}B(v-u,\si)
\left[G(u')F(v')-G(u)F(v)\right]\,d\si
\,d u.
\end{equation}
Here $ u' $ and $ v' $ satisfy
\begin{equation}\label{velocity}
\left\{\begin{split}
u'=\frac{v+u}{2}-\frac{|v-u|}{2}\si,\\
v'=\frac{v+u}{2}+\frac{|v-u|}{2}\si,
\end{split}\right.\quad \si\in\S^2.
\end{equation}
The collision kernel is assumed to take the form of
\begin{equation*}
B(v-u,\si)=C_B |v-u|^{\ga}b(\cos\theta),
\end{equation*}
where $C_B>0$ is a constant, $\cos\theta=\si\cdot (v-u)/|v-u|$, and we assume 
\begin{equation*}
-3<\gamma\leq 1,
\end{equation*} 
and  
\begin{equation*}
\frac{1}{C_b\theta^{1+2s}}\leq \sin\theta b(\cos\theta)\leq \frac{C_b}{\theta^{1
		+2s}},\quad 
		0<s<1,
\end{equation*}
for a constant $ C_{b}>0 $. In the rest of this paper, we use the terminology \textit{hard} potentials if $ \gamma+2s\geq0 $ and \textit{soft} potentials if $ \gamma+2s<0 $.

The Boltzmann equation is one of the fundamental models in collisional kinetic theory. There have been extensive works on the mathematical study of the Boltzmann equation in different aspects. Here we focus on the global existence and large time behavior of solutions in the non-cutoff case. Some basic properties of the Boltzmann collision operator without cutoff were studied in Pao \cite{Pao}. The local-in-time smooth solution in Gevrey class was constructed by Ukai \cite{Ukai84}. The global weak solutions for general data was obtained by Arkeryd \cite{Ark} in the spatially homogeneous case and by Alexandre-Villani \cite{AV02} in the spatially inhomogeneous case. Although Ukai \cite{Ukai74} in 1974 gave the first result of unique global existence of solutions near global Maxwellians for the Boltzmann equation under the Grad's cutoff assumption, it is more than thirty years later that AMUXY \cite{AMUXY,AMUXY11} and Gressman-Strain \cite{GS} independently obtained the global well-posedness theory in the new perturbation frame for the Boltzmann equation without angular cutoff. Those works \cite{AMUXY,AMUXY11,GS} are based on velocity regularity of the collision operator Alexandre-Desvillettes-Villani-Wennberg \cite{ADVB}, coercivity estimate on the linearized operator Mouhot \cite{Mo} and Mouhot-Strain\cite{MSt} and the energy method Liu-Yu \cite{LY04}, Liu-Yang-Yu \cite{LYY04} and Guo \cite{Guo04,G02}. Whenever initial data are close to global Maxwellians with an algebraic large-velocity decay, motivated by a fundamental progress Gualdani-Mischler-Mouhot \cite{GMM} in the cutoff case, there have appeared many important results on global existence for the non-cutoff Boltzmann equation in the perturbation framework, for instance, Alonso-Morimoto-Sun-Yang \cite{AMSY}, He-Jiang \cite{HJ} and H\'erau-Tonon-Tristani \cite{HTT}. For the large time behavior of solutions, we would only mention those results in the soft potential cases, for instance, Caflisch \cite{C}, Strain-Guo \cite{SG1,SG2} and Strain \cite{S}.

The current work is inspired by a recent interesting work Duan-Liu-Sakamoto-Strain \cite{DLSS} on the global existence of lower regularity solutions to the Boltzmann equation without angular cutoff. In fact, it is a big open problem to characterize the optimal mathematical space of initial data with lower regularity in space and velocity variables such that unique solutions may exist globally in time, cf.~\cite{Vi02}. The authors in \cite{DLSS} introduced a new function space such that the norm in $A(\T^3)$ with respect to space variables is finite, where $A(\T^3)$ is the Wiener algebra denoting the set of functions with Fourier series in $\ell^1$. Our goal in this paper is to extend $\ell^1$ to $\ell^p$ for $p>1$ by including some additional space derivatives.

In what follows, we reformulate the Cauchy problem \eqref{Beq} and \eqref{Bid} in the perturbation framework in the usual way.  We recall that the global Maxwellian equilibrium is 
\begin{equation*}
\mu(v)=(2\pi)^{-\frac{3}{2}}e^{-\frac{|v|^{2}}{2}}.
\end{equation*} 
We seek for solutions of the form
\begin{equation*}
F(t,x,v)=\mu+{\mu}^{\frac{1}{2}}f(t,x,v).
\end{equation*}
Plug it into (\ref{Beq}) and (\ref{Bid}) to get
\begin{equation}\label{LBeq}
\pa_tf+v\cdot\na_xf
+Lf=\Ga(f,f)
\end{equation}
with initial data
\begin{equation}\label{idf}
f(0,x,v)=f_0(x,v):=\mu^{-1/2}[F_0(x,v)-\mu].
\end{equation}
Here, in terms of \eqref{bBop}, we have denoted the linearized collision operator $ L $ as
\begin{equation}\label{Ldef}
Lf=-\mu^{-\frac{1}{2}}\left\{Q(\mu,\mu^{\frac{1}{2}}f)+Q(\mu^{\frac{1}{2}}f,\mu)\right\}
\end{equation}
and the nonlinear collision operator as 
\begin{equation*}
\Gamma(f,f)=\mu^{-\frac{1}{2}}Q(\mu^{\frac{1}{2}}f,\mu^{\frac{1}{2}}f).
\end{equation*}
Furthermore, we assume that the following conservation laws
\begin{eqnarray}
&&\int_{\T^3}\int_{\R^3} \mu^{\frac{1}{2}}f(t,x,v)\,dv\,dx=0,\label{pt.id.cl.1}\\
&&\int_{\T^3}\int_{\R^3} v_i\mu^{\frac{1}{2}}f(t,x,v)\,dv\,dx=0,\quad i=1,2,3,\label{pt.id.cl.2}\\
&&\int_{\T^3}\int_{\R^3}|v|^2 \mu^{\frac{1}{2}}f(t,x,v)\,dv\,dx=0,\label{pt.id.cl.3}
\end{eqnarray}
hold true for any $t\geq 0$, and, in particular, they are satisfied for initial data $f_0(x,v)$ at $t=0$. In sum, we are going to study the Cauchy problem \eqref{LBeq} and (\ref{idf}) with (\ref{pt.id.cl.1}), (\ref{pt.id.cl.2}) and (\ref{pt.id.cl.3}). 

Next, motivated by \cite{DLSS,LS,MS}, we introduce function spaces and norms used in this paper.  First of all, we define
\begin{equation*}
\|f\|_{W^{\zeta,p}_{k}L^{\infty}_{T}L^{2}_{v}}=\{\int_{\Z^{3}_{k}}\langle k \rangle^{\zeta p}\sup_{t\in[0,T]}\|\hat{f}(t,k,\cdot)\|^{p}_{L^{2}_{v}}\,d\,\Sigma(k)\}^{\frac{1}{p}}.
\end{equation*}
Here the Fourier transformation is taken with respect to $ x\in \T^{3} $ and the notation $ \int_{\Z^{3}_{k}}(\cdot)\,d\,\Sigma(k) $ means integration over the counting measure on $ \Z^{3} $. In the rest of this paper, we always assume that $p>1$ and $ \zeta>3(1-\frac{1}{p}) $.
\begin{remark}
We expect our solution space to have the property 
$$ 
\|fg\|_{X}\leq C \|f\|_{X}\|g\|_{X}.
$$
For the Sobolev space $W^{\zeta,p}$, we can embed it into $ L^{\infty} $ provided $ \zeta>3(1-\frac{1}{p}) $. Thus, under the assumption that $ \zeta>3(1-\frac{1}{p}) $, this property holds. This is why we propose the assumption $ \zeta>3(1-\frac{1}{p}) $ and it will play an important role as we proceed.
\end{remark}
To treat the long-time behavior in the soft potential case, it is necessary to introduce the following exponential velocity weight. As in \cite{DLSS,DLYZ-VMB}, 
We set
\begin{equation}\label{dvw}
w_{q,\theta}:=e^{q\langle v \rangle^{\theta}},
\end{equation}
where $ \langle v \rangle $ denotes $ \sqrt{1+|v|^{2}} $ and $ q $, $ \theta $ satisfy
\begin{equation}\label{vw}
\left\{\begin{aligned}
&\text{if $ \gamma+2s\geq0 $, then $ q=0 $};\\
&\text{if $-3<\gamma<-2s$, then $ q>0 $ and $ \theta=1 $.}
\end{aligned}\right.
\end{equation}
Thus, we can naturally define the following velocity-weighted norm:
\begin{equation*}
\|w_{q,\theta}f\|_{W^{\zeta,p}_{k}L^{\infty}_{T}L^{2}_{v}}:=\{\int_{\Z^{3}_{k}}\langle k \rangle^{\zeta p}\sup_{t\in[0,T]}\|w_{q,\theta}\hat{f}(t,k,\cdot)\|^{p}_{L^{2}_{v}}\,d\,\Sigma(k)\}^{\frac{1}{p}}.
\end{equation*}
Next, we define the useful velocity weighted $D$-norm
\begin{equation*}
\begin{aligned}
&|w_{q,\theta}f|_{D}^{2}\\
:=&\int_{\R^{3}_{v}}\int_{\R^{3}_{u}}\int_{\S^{2}}B(v-u,\sigma)w^{2}_{q,\theta}(v)\mu(u)(f(v')-f(v))\overline{(f(v')-f(v))}\,d\sigma\,du\,dv\\
&+\int_{\R^{3}_{v}}\int_{\R^{3}_{u}}\int_{\S^{2}}B(v-u,\sigma)w^{2}_{q,\theta}(v)f(u)\overline{f(u)}(\mu^{\frac{1}{2}}(v')-\mu^{\frac{1}{2}}(v))^{2}\,d\sigma\,du\,dv,
\end{aligned}
\end{equation*}
where $ u' $ and $ v' $ are the determined in (\ref{velocity}).
Then we define the following weighted dissipation rate functional
\begin{equation*}
\|w_{q,\theta}f\|_{W^{\zeta,p}_{k}L^{2}_{T}L^{2}_{v,D}}:=\{\int_{\Z^{3}_{k}}\langle k \rangle^{\zeta p}(\int_{0}^{T}|w_{q,\theta}f(t,k)|^{2}_{D}\,dt)^{\frac{p}{2}}\,d\,\Sigma(k)\}^{\frac{1}{p}}.
\end{equation*}
The following notation is introduced to describe the long-time decay rate.
\begin{equation}\label{kappa}
\kappa=\left\{\begin{aligned}
&\text{1}&\text{for $q=0$ and $ \gamma+2s\geq0 $};\\
&\text{$\frac{1}{1+|\gamma+2s|}$}&\text{for $ q>0 $ and $ -3<\gamma<-2s $}.
\end{aligned}\right.
\end{equation}
In the rest of this article, $ C $ denotes a positive constant. $ A\lesssim B $ means that there exists a positive constant $ C $ such that $ A\leq CB $.

Now we are able to state the main result of this article.

\begin{theorem}\label{theorem1}
Let $ p>1 $ and $ \zeta>3(1-\frac{1}{p}) $. Consider problem {\rm(\ref{LBeq})} and {\rm(\ref{idf})} with {\rm(\ref{pt.id.cl.1})}, {\rm(\ref{pt.id.cl.2})} and {\rm(\ref{pt.id.cl.3})} in a torus domain. Assume $ w_{q,\theta} $ satisfies {\rm(\ref{vw})}. There exists $ \epsilon_{0}>0 $ such that if $ \mu+\mu^{\frac{1}{2}}f_{0}(x,v)\geq0 $ and 
\begin{equation*}
\|w_{q,\theta}f_{0}\|_{W^{\zeta,p}_{k}L^{2}_{v}}\leq\epsilon_{0},
\end{equation*}
then there exists a unique global solution $ f(t,x,v) $, $ t>0 $, $ x\in \T^{3} $, $ v\in \R^{3} $ to the above problem, which satisfies that $ \mu+\mu^{\frac{1}{2}}f(t,x,v)\geq0 $ and the uniform estimate
\begin{equation}\label{ne}
\|w_{q,\theta}f\|_{W^{\zeta,p}_{k}L^{\infty}_{T}L^{2}_{v}}+\|w_{q,\theta}f\|_{W^{\zeta,p}_{k}L^{2}_{T}L^{2}_{v,D}}\lesssim\|w_{q,\theta}f_{0}\|_{W^{\zeta,p}_{k}L^{2}_{v}}
\end{equation}
for any $ T>0 $. Moreover, if $ \kappa $ satisfies {\rm(\ref{kappa})}, then there exists a $ \lambda>0 $ such that the solution satisfies the decay rate
\begin{gather}\label{thm1.decay}
\|f\|_{W^{\zeta,p}_{k}L^{2}_{v}}\lesssim e^{-\lambda t^{\kappa}}\|w_{q,\theta}f_{0}\|_{W^{\zeta,p}_{k}L^{2}_{v}}.
\end{gather}
\end{theorem}

The rest of this paper is devoted to the proof of Theorem \ref{theorem1}. 

\section{The main estimates}
In this part, we will prove the main estimates used in this paper. We first introduce the macro-micro decomposition used in this section. We decompose $ f $ into $ \FP f $ and $ \{\FI-\FP\} f $, where
\begin{equation}\label{mmd}
\FP f=\{a+b\cdot v+(|v|^{2}-3)c\}\mu^{\frac{1}{2}}.
\end{equation}
We further use $ [a,b,c] $ to denote the vector defined by (\ref{mmd}). Now let us recall a lemma.
\begin{lemma}\label{lmmd}
$ L $ is the linearized collision operator {\rm(\ref{Ldef})}. Let $ 0<s<1 $ and $ \gamma>-3 $. Then there exists a uniform constant $ C>0 $ such that
\begin{equation*}
\frac{1}{C}|\{\FI-\FP\}g|_{D}^{2}\leq (Lg,g)_{L^{2}_{v}}\leq C|\{\FI-\FP\}g|^{2}_{D}.
\end{equation*}
\end{lemma}
This lemma is proved in \cite{AMUXY}. Using this lemma, we could prove the following result.

\begin{lemma}\label{lemma1}
	The following estimate holds for any $ \eta>0 $ uniformly:
	\begin{equation}\label{lemma1.p1}
	\begin{split}
	&\|f\|_{W^{\zeta,p}_{k}L^{\infty}_{T}L^{2}_{v}}+\|\{\text{$ \FI-\FP $}\}f\|_{W^{\zeta,p}_{k}L^{2}_{T}L^{2}_{v,D}}\\	&\lesssim\|f_{0}\|_{W^{\zeta,p}_{k}L^{2}_{v}}+C_{\eta}\|f\|_{W^{\zeta,p}_{k}L^{\infty}_{T}L^{2}_{v}}\|f\|_{W^{\zeta,p}_{k}L^{2}_{T}L^{2}_{v,D}}+\eta\|\{\text{$ \FI-\FP $}\}f\|_{W^{\zeta,p}_{k}L^{2}_{T}L^{2}_{v,D}}.
	\end{split}
	\end{equation}
\end{lemma}
\begin{proof}
Taking Fourier Transform with respect to $ x $ variable to the equation
\begin{equation}\nonumber
\pa_tf+v\cdot\na_xf
+Lf=\Ga(f,f),
\end{equation}
we have 
\begin{equation}\label{feq}
\pa_t\hat{f}(t,k,v)+iv\cdot\hat{f}(t,k,v)
+L\hat{f}(t,k,v)=\hat{\Ga}(\hat{f},\hat{f})(t,k,v),
\end{equation}
where
\begin{equation}\nonumber
\begin{aligned}
&\hat{\Ga}(\hat{f},\hat{g})(k,v)\\
&=\int_{\R^{3}}\int_{\S^{2}}B(v-u,\sigma)\mu^{\frac{1}{2}}([\hat{f}(u')*\hat{g}(v')](k)-[\hat{f}(u)*\hat{g}(v)](k))\,d\sigma\,du.
\end{aligned}
\end{equation}
The convolution above is taken with respect to $ k\in\Z^{3} $. Taking product with $ \bar{\hat{f}} $ and taking real part, we have
\begin{equation}\nonumber
\frac{1}{2}\frac{d}{dt}|\hat{f}|^{2}+\Re(L\hat{f},\hat{f})=\Re(\Ga(\hat{f},\hat{f}),\hat{f}).
\end{equation}
Here $ (\cdot,\cdot) $ denotes the inner product in the complex plane. We then integrate with respect to $ v $ and then $ t $ to deduce 
\begin{equation}\nonumber
\frac{1}{2}\|\hat{f}\|^{2}_{L^{2}_{v}}+\int_{0}^{t}\Re(L\hat{f},\hat{f})_{L^{2}_{v}}\,d\tau=\frac{1}{2}\|\hat{f_{0}}\|^{2}_{L^{2}_{v}}+\int_{0}^{t}\Re(\Ga(\hat{f},\hat{f}),\hat{f})_{L^{2}_{v}}\,d\tau.
\end{equation}
Applying Lemma \ref{lmmd}, we have
\begin{equation}\nonumber
\frac{1}{2}\|\hat{f}\|^{2}_{L^{2}_{v}}+\frac{1}{C}\int_{0}^{t}|\{\textbf{I}-\textbf{P}\}\hat{f}|_{D}^{2}\,d\tau\leq\frac{1}{2}\|\hat{f_{0}}(k,\cdot)\|_{L^{2}_{v}}^{2}+\int_{0}^{t}|\Re(\hat{\Gamma}(\hat{f},\hat{f}),\hat{f})_{L^{2}_{v}}|\,d\tau.
\end{equation}
When $ p>1 $, we could prove that there exists a constant $ C $ such that for any non-negative numbers $ A $ and $ B $, it holds that
\begin{equation}\nonumber
\frac{1}{C}(A^{p}+B^{p})^{\frac{2}{p}} \leq A^{2}+B^{2}\leq C(A^{p}+B^{p})^{\frac{2}{p}}.
\end{equation}
Applying this inequality, we can deduce that there exists a constant $ C>0 $ such that 
\begin{equation}\nonumber
\begin{split}
&\|\hat{f}(t,k,\cdot)\|^{p}_{L^{2}_{v}}+(\int_{0}^{t}|\{\textbf{I}-\textbf{P}\}\hat{f}|_{D}^{2}\,d\tau)^{\frac{p}{2}}\\
&\leq C\{\|\hat{f}_{0}(k,\cdot)\|^{p}_{L^{2}_{v}}+(\int_{0}^{t}|\Re(\hat{\Gamma}(\hat{f},\hat{f}),\hat{f})_{L^{2}_{v}}|\,d\tau)^{\frac{p}{2}}\}.
\end{split}
\end{equation}
We take supreme with respect to $ t $ in $ [0,T] $ and take product with $ \left\langle k \right\rangle^{\zeta p} $ to get
\begin{equation}\nonumber
\begin{split}
&\left\langle k \right\rangle^{\zeta p}\sup_{t\in[0,T]}\|\hat{f}(t,k,\cdot)\|^{p}_{L^{2}_{v}}+\left\langle k \right\rangle^{\zeta p}(\int_{0}^{T}|\{\textbf{I}-\textbf{P}\}\hat{f}|_{D}^{2}\,d\tau)^{\frac{p}{2}}\\
&\leq C\{\|\left\langle k \right\rangle^{\zeta}\hat{f}_{0}(k,\cdot)\|^{p}_{L^{2}_{v}}+\left\langle k \right\rangle^{\zeta p}(\int_{0}^{T}|\Re(\hat{\Gamma}(\hat{f},\hat{f}),\hat{f})_{L^{2}_{v}}|\,d\tau)^{\frac{p}{2}}\}.
\end{split}
\end{equation}
Integating with respect to $ k $ over $ \mathbb{Z}_{k}^{3} $, we get
\begin{equation}\label{ineq1}
\begin{split}
&\|f\|_{W^{\zeta,p}_{k}L^{\infty}_{T}L^{2}_{v}}^{p}+\int_{\mathbb{Z}^{3}_{k}}(\int_{0}^{T}|\left\langle k \right\rangle^{\zeta}\{\textbf{I}-\textbf{P}\}\hat{f}|_{D}^{2}\,d\tau)^{\frac{p}{2}}\,d\,\Sigma(k)\\
&\leq C\{\|f_{0}\|_{W^{\zeta,p}_{k}L^{2}_{v}}^{p}+\int_{\mathbb{Z}^{3}_{k}}\left\langle k \right\rangle^{\zeta p}(\int_{0}^{T}|\Re(\hat{\Gamma}(\hat{f},\hat{f}),\hat{f})_{L^{2}_{v}}|\,d\tau)^{\frac{p}{2}}\,d\,\Sigma(k)\}.
\end{split}
\end{equation}
We further have the equality $ (\hat{\Gamma}(\hat{f},\hat{f}),\hat{f})_{L^{2}_{v}}=(\hat{\Gamma}(\hat{f},\hat{f}),\{\textbf{I}-\textbf{P}\}\hat{f})_{L^{2}_{v}} $. By \cite[Lemma 3.2]{DLSS}, we have 
\begin{equation}\nonumber
|(\hat{\Gamma}(\hat{f},\hat{g})(k),\hat{h}(k))_{L^{2}_{v}}|\leq C\int_{\mathbb{Z}^{3}_{l}}\|\hat{f}(k-l)\|_{L^{2}_{v}}|\hat{g}(l)|_{D}|\hat{h}(k)|_{D}\,d\,\Sigma(l).
\end{equation}
Applying the above inequality, we have
\begin{equation}\label{ineq5}
\begin{split}
&\int_{\mathbb{Z}^{3}_{k}}\left\langle k \right\rangle^{\zeta p}(\int_{0}^{T}|\Re(\hat{\Gamma}(\hat{f},\hat{f}),\{\textbf{I}-\textbf{P}\}\hat{f})_{L^{2}_{v}}|\,d\tau)^{\frac{p}{2}}\,d\,\Sigma(k)\\
&\leq C\int_{\mathbb{Z}^{3}_{k}}\left\langle k \right\rangle^{\zeta p}\{\int_{0}^{T}(|\{\textbf{I}-\textbf{P}\}\hat{f}(k)|_{D})(\int_{\mathbb{Z}^{3}_{l}}\|\hat{f}(k-l)\|_{L^{2}_{v}}|\hat{f}(l)|_{D}\,d\,\Sigma(l))d\tau\}^{\frac{p}{2}}\,d\,\Sigma(k).
\end{split}
\end{equation}
Thus, it suffices to estimate
\begin{equation}\label{ineq}
\int_{\mathbb{Z}^{3}_{k}}\left\langle k \right\rangle^{\zeta p}\{\int_{0}^{T}(|\{\textbf{I}-\textbf{P}\}\hat{f}(\tau,k)|_{D})(\int_{\mathbb{Z}^{3}_{l}}\|\hat{f}(\tau,k-l)\|_{L^{2}_{v}}|\hat{f}(\tau,l)|_{D}\,d\,\Sigma(l))d\tau\}^{\frac{p}{2}}\,d\,\Sigma(k).
\end{equation}	
First we can apply Cauchy Inequality with respect to $ \int_{0}^{T}(\cdot)\,d\tau $ to obtain
\begin{equation*}
\begin{split}
\int_{\mathbb{Z}^{3}_{k}}\left\langle k \right\rangle^{\zeta p}&\{\int_{0}^{T}(|\{\textbf{I}-\textbf{P}\}\hat{f}(\tau,k)|_{D})(\int_{\mathbb{Z}^{3}_{l}}\|\hat{f}(\tau,k-l)\|_{L^{2}_{v}}|\hat{f}(\tau,l)|_{D}\,d\,\Sigma(l))d\tau\}^{\frac{p}{2}}\,d\,\Sigma(k)\\
\leq C\int_{\mathbb{Z}^{3}_{k}}&\{\left\langle k \right\rangle^{\frac{\zeta p}{2}}(\int_{0}^{T}|\{\textbf{I}-\textbf{P}\}\hat{f}(\tau,k)|_{D}^{2}\,d\tau)^{\frac{p}{4}}\}\\
&\times \{\left\langle k \right\rangle^{\frac{\zeta p}{2}}[\int_{0}^{T}(\int_{\mathbb{Z}^{3}_{l}}\|\hat{f}(\tau,k-l)\|_{L^{2}_{v}}|\hat{f}(\tau,l)|_{D}\,d\,\Sigma(l))^{2}\,d\tau]^{\frac{p}{4}}\}\,d\,\Sigma(k)
\end{split}
\end{equation*}
It further follows by Young's Inequality with $ \eta>0 $ that
\begin{equation}\label{ineq2}
\begin{split}
\int_{\mathbb{Z}^{3}_{k}}\left\langle k \right\rangle^{\zeta p}&\{\int_{0}^{T}(|\{\textbf{I}-\textbf{P}\}\hat{f}(\tau,k)|_{D})(\int_{\mathbb{Z}^{3}_{l}}\|\hat{f}(\tau,k-l)\|_{L^{2}_{v}}|\hat{f}(\tau,l)|_{D}\,d\,\Sigma(l))d\tau\}^{\frac{p}{2}}\,d\,\Sigma(k)\\
\leq C\eta\int_{\mathbb{Z}^{3}_{k}}&(\int_{0}^{T}|\left\langle k \right\rangle^{\zeta}\{\textbf{I}-\textbf{P}\}\hat{f}(\tau,k)|_{D}^{2}\,d\tau)^{\frac{p}{2}}\,d\,\Sigma(k)\\
+\frac{C}{4\eta}&\int_{\mathbb{Z}^{3}_{k}}\{\int_{0}^{T}(\int_{\mathbb{Z}^{3}_{l}}\left\langle k \right\rangle^{\zeta}\|\hat{f}(\tau,k-l)\|_{L^{2}_{v}}|\hat{f}(\tau,l)|_{D}\,d\,\Sigma(l) )^{2}\,d\tau\}^{\frac{p}{2}}\,d\,\Sigma(k).
\end{split}
\end{equation}	
Now we estimate the latter one of the above. Notice that 
\begin{equation}\label{ei}
\begin{split} 
\left\langle k \right\rangle^{h}\leq& \{(1+|k-l|)+(1+|l|)\}^{h}\\
\leq& 2^{h}\{(1+|k-l|)^{h}+(1+|l|)^{h}\}
=C(\left\langle k-l \right\rangle^{h}+\left\langle l \right\rangle^{h}).
\end{split}
\end{equation}
holds for any $ h\geq0 $. Applying (\ref{ei}) to the term $ \langle k \rangle^{\zeta} $, we have 
\begin{equation}\label{ineq0}
\begin{split}
&\int_{\mathbb{Z}^{3}_{k}}\{\int_{0}^{T}(\int_{\mathbb{Z}^{3}_{l}}\left\langle k \right\rangle^{\zeta}\|\hat{f}(\tau,k-l)\|_{L^{2}_{v}}|\hat{f}(\tau,l)|_{D}\,d\,\Sigma(l))^{2}\,d\tau\}^{\frac{p}{2}}\,d\,\Sigma(k)\\
\leq&C\int_{\mathbb{Z}^{3}_{k}}\{\int_{0}^{T}(\int_{\mathbb{Z}^{3}_{l}}\left\langle k-l \right\rangle^{\zeta}\|\hat{f}(\tau,k-l)\|_{L^{2}_{v}}|\hat{f}(\tau,l)|_{D}\,d\,\Sigma(l) )^{2}\,d\tau\}^{\frac{p}{2}}\,d\,\Sigma(k)\\
&+C\int_{\mathbb{Z}^{3}_{k}}\{\int_{0}^{T}(\int_{\mathbb{Z}^{3}_{l}}\left\langle l \right\rangle^{\zeta}\|\hat{f}(\tau,k-l)\|_{L^{2}_{v}}|\hat{f}(\tau,l)|_{D}\,d\,\Sigma(l) )^{2}\,d\tau\}^{\frac{p}{2}}\,d\,\Sigma(k).
\end{split}
\end{equation}
For the former term of (\ref{ineq0}), we apply Minkowski's inequality to deduce the following estimate
\begin{equation}\nonumber
\begin{split}
&\int_{\mathbb{Z}^{3}_{k}}\{\int_{0}^{T}(\int_{\mathbb{Z}^{3}_{l}}\left\langle k-l \right\rangle^{\zeta}\|\hat{f}(\tau,k-l)\|_{L^{2}_{v}}|\hat{f}(\tau,l)|_{D}\,d\,\Sigma(l) )^{2}\,d\tau\}^{\frac{p}{2}}\,d\,\Sigma(k)\\
&\leq\int_{\Z^{3}_{k}}\{\int_{\Z^{3}_{l}}(\int_{0}^{T}\langle k-l \rangle ^{2\zeta}\|\hat{f}(\tau,k-l)\|^{2}_{L^{2}_{v}}|\hat{f}(\tau,l)|^{2}_{D}\,d\tau)^{\frac{1}{2}}\,d\,\Sigma(l)\}^{p}\,d\,\Sigma(k)\\
&\leq\int_{\Z^{3}_{k}}[\int_{\Z^{3}_{l}}\langle k-l \rangle ^{\zeta}\sup_{\tau\in[0,T]}{\|\hat{f}(\tau,k-l)\|_{L^{2}_{v}}}(\int_{0}^{T}|\hat{f}(\tau,l)|^{2}_{D}\,d\tau)^{\frac{1}{2}}\,d\,\Sigma(l)]^{p}\,d\,\Sigma(k),
\end{split}
\end{equation}
which is further bounded by Minkowski's inequality as 

\begin{equation}\nonumber
\begin{split}
&=\{\int_{\Z^{3}_{k}}[\int_{\Z^{3}_{l}}\langle k-l \rangle ^{\zeta}\sup_{\tau\in[0,T]}{\|\hat{f}(\tau,k-l)\|_{L^{2}_{v}}}(\int_{0}^{T}|\hat{f}(\tau,l)|^{2}_{D}\,d\tau)^{\frac{1}{2}}\,d\,\Sigma(l)]^{p}\,d\,\Sigma(k)\}^{\frac{1}{p}\cdot p}\\
&\leq \{\int_{\Z^{3}_{l}}[\int_{\Z^{3}_{k}}\langle k-l \rangle ^{\zeta p}\sup_{\tau\in[0,T]}{\|\hat{f}(\tau,k-l)\|^{p}_{L^{2}_{v}}}(\int_{0}^{T}|\hat{f}(\tau,l)|^{2}_{D}\,d\tau)^{\frac{p}{2}}\,d\,\Sigma(k)]^{\frac{1}{p}}\,d\,\Sigma(l)\}^{p}\\
&=\|(\int_{0}^{T}|\hat{f}(\tau,l)|^{2}_{D}d\tau)^{\frac{1}{2}}\|_{l^{1}_{l}}^{p}\times\|f\|_{W_k^{\zeta,p}L^\infty_TL^2_v}^{p}.
\end{split}
\end{equation}
For the latter one of (\ref{ineq0}), similarly, we can use the same method to deduce that
\begin{equation}\nonumber
\begin{split}
&\int_{\mathbb{Z}^{3}_{k}}\{\int_{0}^{T}(\int_{\mathbb{Z}^{3}_{l}}\left\langle l \right\rangle^{\zeta}\|\hat{f}(\tau,k-l)\|_{L^{2}_{v}}|\hat{f}(\tau,l)|_{D}\,d\,\Sigma(l) )^{2}\,d\tau\}^{\frac{p}{2}}\,d\,\Sigma(k)\\
&\leq \|\sup_{\tau\in[0,T]}{\|\hat{f}(\tau,l)\|_{L^{2}_{v}}}\|_{l^{1}_{l}}^{p}\times\|f\|_{W_k^{\zeta,p}L^2_TL^2_v,D}^{p}.
\end{split}
\end{equation}
Thus, we can deduce that 
\begin{equation}\nonumber
\begin{split}
&\int_{\mathbb{Z}^{3}_{k}}\{\int_{0}^{T}(\int_{\mathbb{Z}^{3}_{l}}\left\langle k \right\rangle^{\zeta}\|\hat{f}(\tau,k-l)\|_{L^{2}_{v}}|\hat{f}(\tau,l)|_{D}\,d\,\Sigma(l))^{2}\,d\tau\}^{\frac{p}{2}}\,d\,\Sigma(k)\\
\leq& C\|(\int_{0}^{T}|\hat{f}(\tau,l)|^{2}_{D}d\tau)^{\frac{1}{2}}\|_{l^{1}_{l}}^{p}\times\|f\|_{W_k^{\zeta,p}L^\infty_TL^2_v}^{p}\\
&+C\|\sup_{\tau\in[0,T]}{\|\hat{f}(\tau,l)\|_{L^{2}_{v}}}\|_{l^{1}_{l}}^{p}\times\|f\|_{W_k^{\zeta,p}L^2_TL^2_v,D}.
\end{split}
\end{equation}

To get the anticipated result, it suffices to bound 
\begin{center}
$ \|(\int_{0}^{T}|\hat{f}(\tau,l)|^{2}_{D}d\tau)^{\frac{1}{2}}\|_{l^{1}_{l}} $ and $ \|\sup_{\tau\in[0,T]}{\|\hat{f}(\tau,l)\|_{L^{2}_{v}}}\|_{l^{1}_{l}} $.
\end{center} 
Bounding them by $ \|f\|_{W_k^{\zeta,p}L^2_TL^2_{v,D}} $ and $ \|f\|_{W_k^{\zeta,p}L^\infty_TL^2_v} $ is the goal to accomplish. 
For the first one,  
by H\"older's Inequality with $ \frac{1}{p}+\frac{1}{p'}=1 $, we can deduce that
\begin{equation}\nonumber
\begin{split}
&\int_{\Z^{3}_{l}}(\int_{0}^{T}|\hat{f}(\tau,l)|^{2}_{D}\,d\tau)^{\frac{1}{2}}\,d\,\Sigma(l)\\
&=\int_{\Z^{3}_{l}}(\int_{0}^{T}|\langle l \rangle ^{\zeta}\hat{f}(\tau,l)|^{2}_{D}d\tau)^{\frac{1}{2}}\langle l \rangle^{-\zeta}\,d\,\Sigma(l)\\
&\leq (\int_{\Z^{3}_{l}}(\int_{0}^{T}|\langle l \rangle ^{\zeta}\hat{f}(\tau,l)|^{2}_{D}\,d\tau)^{\frac{p}{2}}\,d\,\Sigma(l))^{\frac{1}{p}}(\int_{\Z^{3}_{l}}\langle l \rangle ^{-\zeta p'}\,d\,\Sigma(l))^{\frac{1}{p'}}\\
&=\|f\|_{W_k^{\zeta,p}L^2_TL^2_{v,D}}(\int_{\Z^{3}_{l}}\langle l \rangle ^{-\zeta p'}\,d\,\Sigma(l))^{\frac{1}{p'}}.
\end{split}
\end{equation}
Where $ p $ and $ p' $ are conjugate, say, $ p'=\frac{p}{p-1} $.

Since $ \zeta>3(1-\frac{1}{p}) $, we have $ \zeta p'=\frac{\zeta p}{p-1}>3 $. Thus $ \int_{\Z^{3}_{l}}\langle l \rangle ^{-\zeta p'}\,d\,\Sigma(l) $ is finite, which implies that $ \|(\int_{0}^{T}|\hat{f}(\tau,l)|^{2}_{D}\,d\tau)^{\frac{1}{2}}\|_{l^{1}_{l}}\leq C\|f\|_{W_k^{\zeta,p}L^2_TL^2_v,D} $.

For the latter one, we can similarly deduce that
\begin{equation} \nonumber
\begin{split}
&\int_{\Z^{3}_{l}}\sup_{\tau\in[0,T]}{\|\hat{f}(\tau,l)\|_{L^{2}_{v}}}\,d\,\Sigma(l)\\
&=\int_{\Z^{3}_{l}}\sup_{\tau\in[0,T]}{\|\langle l \rangle ^{\zeta}\hat{f}(\tau,l)\|_{L^{2}_{v}}}\langle l \rangle ^{-\zeta}\,d\,\Sigma(l)\\
&\leq (\int_{\Z^{3}_{l}}\sup_{\tau\in[0,T]}{\|\langle l \rangle ^{\zeta}\hat{f}(\tau,l)\|^{p}_{L^{2}_{v}}}\,d\,\Sigma(l))^{\frac{1}{p}}(\int_{\Z^{3}_{l}}\langle l \rangle ^{-\zeta p'}\,d\,\Sigma(l))^{\frac{1}{p'}}\\
&\leq C\|f\|_{W_k^{\zeta,p}L^\infty_TL^2_v}.
\end{split}
\end{equation}
Thus we can deduce that 
\begin{equation}\label{ineq3}
\begin{split}
&\int_{\mathbb{Z}^{3}_{k}}\{\int_{0}^{T}(\int_{\mathbb{Z}^{3}_{l}}\left\langle k \right\rangle^{\zeta}\|\hat{f}(\tau,k-l)\|_{L^{2}_{v}}|\hat{f}(\tau,l)|_{D}\,d\,\Sigma(l) )^{2}\,d\tau\}^{\frac{p}{2}}\,d\,\Sigma(k)\\
&\leq C\|f\|_{W_k^{\zeta,p}L^\infty_TL^2_v}^{p}\|f\|_{W_k^{\zeta,p}L^2_TL^2_v,D}^{p}.
\end{split}
\end{equation}
Combining (\ref{ineq1}), (\ref{ineq5}), (\ref{ineq2}) and (\ref{ineq3}), we have
\begin{equation}\nonumber
\begin{split}
&\|f\|_{W^{\zeta,p}_{k}L^{\infty}_{T}L^{2}_{v}}^{p}+\|\{\textbf{I}-\textbf{P}\}f\|^{p}_{W^{\zeta,p}_{k}L^{2}_{T}L^{2}_{v,D}}\\
&\lesssim\|f_{0}\|_{W^{\zeta,p}_{k}L^{2}_{v}}^{p}+C_{\eta}\|f\|^{p}_{W^{\zeta,p}_{k}L^{\infty}_{T}L^{2}_{v}}\|f\|^{p}_{W^{\zeta,p}_{k}L^{2}_{T}L^{2}_{v,D}}+\eta\|\{\textbf{I}-\textbf{P}\}f\|^{p}_{W^{\zeta,p}_{k}L^{2}_{T}L^{2}_{v,D}}.
\end{split}
\end{equation}
Noticing that there exists a constant $ C>0 $ such that for any $ D,E,F\geq0 $, we have
\begin{equation}\nonumber
\frac{1}{C}(D+E+F)\leq(D^{p}+E^{p}+F^{p})^{\frac{1}{p}}\leq C(D+E+F).
\end{equation}
So we can finally deduce
\begin{equation}\nonumber
\begin{split}
&\|f\|_{W^{\zeta,p}_{k}L^{\infty}_{T}L^{2}_{v}}+\|\{\textbf{I}-\textbf{P}\}f\|_{W^{\zeta,p}_{k}L^{2}_{T}L^{2}_{v,D}}\\
&\lesssim\|f_{0}\|_{W^{\zeta,p}_{k}L^{2}_{v}}+C_{\eta}\|f\|_{W^{\zeta,p}_{k}L^{\infty}_{T}L^{2}_{v}}\|f\|_{W^{\zeta,p}_{k}L^{2}_{T}L^{2}_{v,D}}+\eta\|\{\textbf{I}-\textbf{P}\}f\|_{W^{\zeta,p}_{k}L^{2}_{T}L^{2}_{v,D}}.
\end{split}
\end{equation}
This proves \eqref{lemma1.p1} and completes the proof of Lemma \ref{lemma1}.
\end{proof}

To proceed, it is necessary to obtain the macroscopic estimate. Applying the results in \cite{DLSS}, we will obtain the macroscopic estimate in the following lemma.

\begin{lemma}\label{lemma2}
	Assume all the assumptions of Theorem {\rm\ref{theorem1}} hold true. It holds that 
	\begin{equation}\label{me}
	\begin{split}
	\|[a,b,c]\|_{W^{\zeta,p}_{k}L^{2}_{T}}\lesssim& \|\{\FI-\FP\}f\|_{W^{\zeta,p}_{k}L^{2}_{T}L^{2}_{v,D}}+\|f\|_{W^{\zeta,p}_{k}L^{\infty}_{T}L^{2}_{v}}+\|f_{0}\|_{W^{\zeta,p}_{k}L^{2}_{v}}\\
	&+(\int_{\mathbb{Z}^{3}_{k}}\left\langle k \right\rangle^{\zeta p}(\int_{0}^{T}|(\hat{\Gamma}(\hat{f},\hat{f}),\mu^{\frac{1}{4}})_{L^{2}_{v}}|^{2}\,d\tau)^{\frac{p}{2}}\,d\,\Sigma(k))^{\frac{1}{p}}.
	\end{split}
	\end{equation}
\end{lemma}
\begin{proof}
 From the proof of \cite[Theorem 5.1]{DLSS}, for arbitrarily small $ \eta_{i}>0(i=1,2) $, we can deduce that
 \begin{equation}\nonumber
	\begin{split}
	\int_{0}^{t}|\hat{c}(\tau,k)|^{2}\,d\tau\lesssim& \|\hat{f}(k,t)\|^{2}_{L^{2}_{v}}+\|\hat{f}_{0}(k)\|^{2}_{L^{2}_{v}}+\eta_{1}\int_{0}^{t}|\hat{b}(\tau,k)|^{2}d\tau\\
	&+C_{\eta_{1}}\int_{0}^{t}|\{\FI-\FP\}\hat{f}(\tau,k)|_{D}^{2}\,d\tau+C_{\eta_{1}}\int_{0}^{t}|(\hat{\Gamma}(\hat{f},\hat{f}),\mu^{\frac{1}{4}})_{L^{2}_{v}}|^{2}\,d\tau;\\
	\int_{0}^{t}|\hat{b}(\tau,k)|^{2}\,d\tau\lesssim&\|\hat{f}(k,t)\|^{2}_{L^{2}_{v}}+\|\hat{f}_{0}(k)\|^{2}_{L^{2}_{v}}+\eta_{2}\int_{0}^{t}|\hat{a}(\tau,k)|^{2}\,d\tau\\
	&+C_{\eta_{2}}\int_{0}^{t}|\hat{c}(\tau,k)|^{2}\,d\tau+C_{\eta_{2}}\int_{0}^{t}|\{\FI-\FP\}\hat{f}(\tau,k)|_{D}^{2}\,d\tau\\
	&+C_{\eta_{2}}\int_{0}^{t}|(\hat{\Gamma}(\hat{f},\hat{f}),\mu^{\frac{1}{4}})_{L^{2}_{v}}|^{2}\,d\tau;\\
	\int_{0}^{t}|\hat{a}(\tau,k)|^{2}\,d\tau\lesssim&\|\hat{f}(k,t)\|^{2}_{L^{2}_{v}}+\|\hat{f}_{0}(k)\|^{2}_{L^{2}_{v}}+\int_{0}^{t}|\hat{b}(\tau,k)|^{2}\,d\tau\\
	&+\int_{0}^{t}|\{\FI-\FP\}\hat{f}(\tau,k)|_{D}^{2}\,d\tau +\int_{0}^{t}|(\hat{\Gamma}(\hat{f},\hat{f}),\mu^{\frac{1}{4}})_{L^{2}_{v}}|^{2}\,d\tau.
	\end{split}
\end{equation} 
	Taking $ \mu_{1} $ and $ \mu_{2} $ to be small enough, we have
	\begin{equation}\nonumber
	\begin{split}
	\int_{0}^{t}|[a,b,c]|^{2}\,d\tau\lesssim& \int_{0}^{t}|\{\FI-\FP\}\hat{f}(\tau,k)|_{D}^{2}\,d\tau+\|\hat{f}(k,t)\|^{2}_{L^{2}_{v}}+\|\hat{f}_{0}(k)\|^{2}_{L^{2}_{v}}\\
	&+\int_{0}^{t}|(\hat{\Gamma}(\hat{f},\hat{f}),\mu^{\frac{1}{4}})_{L^{2}_{v}}|^{2}\,d\tau.
	\end{split}
	\end{equation}
	Taking supreme of $ t $ over $ [0,T] $ on both sides, we further have
	\begin{equation}\nonumber
	\begin{split}
	\|[a,b,c]\|_{L^{2}_{T}}\lesssim&\|\{\FI-\FP\}\hat{f}\|_{L^{2}_{T}L^{2}_{v,D}}+\sup_{t\in[0,T]}\|\hat{f}(k,t)\|_{L^{2}_{v}}+\|\hat{f}_{0}(k)\|_{L^{2}_{v}}\\
	&+(\int_{0}^{T}|(\hat{\Gamma}(\hat{f},\hat{f}),\mu^{\frac{1}{4}})_{L^{2}_{v}}|^{2}\,d\tau)^{\frac{1}{2}}.
	\end{split}
	\end{equation}
	Then we know that
	\begin{equation} \nonumber
	\begin{split}
	\|[a,b,c]\|_{L^{2}_{T}}^{p}\lesssim&\|\{\FI-\FP\}\hat{f}\|_{L^{2}_{T}L^{2}_{v,D}}^{p}+\sup_{t\in[0,T]}\|\hat{f}(k,t)\|_{L^{2}_{v}}^{p}+\|\hat{f}_{0}(k)\|_{L^{2}_{v}}^{p}\\
	&+(\int_{0}^{T}|(\hat{\Gamma}(\hat{f},\hat{f}),\mu^{\frac{1}{4}})_{L^{2}_{v}}|^{2}\,d\tau)^{\frac{p}{2}}.
	\end{split}
	\end{equation} 
	Thus taking product with $ \langle k \rangle^{\zeta p} $ and integrating with respect to $ k $, we can deduce
	\begin{equation}\nonumber
	\begin{split}
	\|[a,b,c]\|_{W^{\zeta,p}_{k}L^{2}_{T}}\lesssim& \|\{\FI-\FP\}f\|_{W^{\zeta,p}_{k}L^{2}_{T}L^{2}_{v,D}}+\|f\|_{W^{\zeta,p}_{k}L^{\infty}_{T}L^{2}_{v}}+\|f_{0}\|_{W^{\zeta,p}_{k}L^{2}_{v}}\\
	&+(\int_{\mathbb{Z}^{3}_{k}}\left\langle k \right\rangle^{\zeta p}(\int_{0}^{T}|(\hat{\Gamma}(\hat{f},\hat{f}),\mu^{\frac{1}{4}})_{L^{2}_{v}}|^{2}\,d\tau)^{\frac{p}{2}}\,d\,\Sigma(k))^{\frac{1}{p}}.
	\end{split}
	\end{equation}
This then completes the proof of Lemma \ref{lemma2}.
\end{proof}

To estimate the last term of inequality (\ref{me}), we further prove the following lemma.

\begin{lemma}\label{lemma3} 
Let $ \gamma+2s\geq 0 $. Assume that all the assumptions of Theorem {\rm\ref{theorem1}} hold true. $ u $ only depends on $ v $ and $ u $ decays rapidly at infinity. The following estimate holds:
	\begin{equation}\label{lemma3.p1}
	\begin{split}
	&(\int_{\mathbb{Z}^{3}_{k}}\left\langle k \right\rangle^{\zeta p}(\int_{0}^{T}|(\hat{\Gamma}(\hat{f},\hat{g}),u(v))_{L^{2}_{v}}|^{2}\,d\tau)^{\frac{p}{2}}\,d\,\Sigma(k))^{\frac{1}{p}}\\
	&\lesssim\|f\|_{W^{\zeta,p}_{k}L^{\infty}_{T}L^{2}_{v}}\|g\|_{W^{\zeta,p}_{k}L^{2}_{T}L^{2}_{v,D}}+\|g\|_{W^{\zeta,p}_{k}L^{\infty}_{T}L^{2}_{v}}\|f\|_{W^{\zeta,p}_{k}L^{2}_{T}L^{2}_{v,D}}.
	\end{split}
	\end{equation}
	The constant only depends on $ u(v) $.
\end{lemma}

\begin{proof}
	Notice that
	\begin{equation}\nonumber
	\begin{split}
	&(\int_{0}^{T}|(\hat{\Gamma}(\hat{f},\hat{g}),u(v))_{L^{2}_{v}}|^{2}\,d\tau)^{\frac{p}{2}}\\
	&=\{\int_{0}^{T}[\int_{\mathbb{R}^{3}_{v}}(\int_{\mathbb{Z}^{3}_{l}}\Gamma(\hat{f}(\tau,k-l),\hat{g}(\tau,l))u(v)\,d\,\Sigma(l))dv]^{2}\,d\tau\}^{\frac{p}{2}}.
	\end{split}
	\end{equation}
	Recall \cite[Lemma 4.1]{DLSS}. We have
	\begin{equation}\nonumber
	\begin{split}
	&(\Gamma(\hat{f}(\tau,k-l),\hat{g}(\tau,l)),u(v))_{L^{2}_{v}}\\
	&\lesssim(\|\hat{f}(\tau,k-l)\|_{L^{2}_{v}}|\hat{g}(\tau,l)|_{D}+\|\hat{g}(\tau,l)\|_{L^{2}_{v}}|\hat{f}(\tau,k-l)|_{D})|u|_{D}.
	\end{split}
	\end{equation}
	Applying the above inequality, we can deduce that
	\begin{equation}\label{ineq4}
	\begin{split}
	&\{\int_{0}^{T}[\int_{\mathbb{R}^{3}_{v}}(\int_{\mathbb{Z}^{3}_{l}}\Gamma(\hat{f}(\tau,k-l),\hat{g}(\tau,l))u(v)\,d\,\Sigma(l))\,dv]^{2}\,d\tau\}^{\frac{p}{2}}\\
	\leq&C\{\int_{0}^{T}(\int_{\mathbb{Z}^{3}_{l}}\|\hat{f}(\tau,k-l)\|_{L^{2}_{v}}|\hat{g}(\tau,l)|_{D}\,d\,\Sigma(l))^{2}\,d\tau\}^{\frac{p}{2}}\\
	&+C\{\int_{0}^{T}(\int_{\mathbb{Z}^{3}_{l}}\|\hat{g}(\tau,l)\|_{L^{2}_{v}}|\hat{f}(\tau,k-l)|_{D}\,d\,\Sigma(l))^{2}\,d\tau\}^{\frac{p}{2}}.
	\end{split}
	\end{equation}
	For the first term of (\ref{ineq4}), we then apply Minkowski's inequality to deduce that
	\begin{equation}\nonumber
	\begin{split}
	&\{\int_{0}^{T}(\int_{\mathbb{Z}^{3}_{l}}\|\hat{f}(\tau,k-l)\|_{L^{2}_{v}}|\hat{g}(\tau,l)|_{D}\,d\,\Sigma(l))^{2}\,d\tau\}^{\frac{p}{2}}\\
	&\leq\{\int_{\mathbb{Z}^{3}_{l}}(\int_{0}^{T}\|\hat{f}(\tau,k-l)\|_{L^{2}_{v}}^{2}|\hat{g}(\tau,l)|^{2}_{D}\,d\tau)^{\frac{1}{2}}\,d\,\Sigma(l)\}^{p}\\
	&\leq \{\int_{\mathbb{Z}^{3}_{l}}(\sup_{\tau\in[0,T]}\|\hat{f}(\tau,k-l)\|_{L^{2}_{v}})(\int_{0}^{T}|\hat{g}(\tau,l)|_{D}^{2}\,d\tau)^{\frac{1}{2}}\,d\,\Sigma(l)\}^{p}.
	\end{split}
	\end{equation}
	Applying the same method to the second term, we finally can deduce that
	\begin{equation}\nonumber
	\begin{split}
	&\{\int_{0}^{T}[\int_{\mathbb{R}^{3}_{v}}(\int_{\mathbb{Z}^{3}_{l}}\Gamma(\hat{f}(\tau,k-l),\hat{g}(\tau,l))u(v)\,d\,\Sigma(l))dv]^{2}\,d\tau\}^{\frac{p}{2}}\\
	\leq&C\{\int_{\mathbb{Z}^{3}_{l}}(\sup_{\tau\in[0,T]}\|\hat{f}(\tau,k-l)\|_{L^{2}_{v}})(\int_{0}^{T}|\hat{g}(\tau,l)|_{D}^{2}\,d\tau)^{\frac{1}{2}}\,d\,\Sigma(l)\}^{p}\\
	&+C\{\int_{\mathbb{Z}^{3}_{l}}(\sup_{\tau\in[0,T]}\|\hat{g}(\tau,l)\|_{L^{2}_{v}})(\int_{0}^{T}|\hat{f}(\tau,k-l)|_{D}^{2}\,d\tau)^{\frac{1}{2}}\,d\,\Sigma(l)\}^{p}.
	\end{split}
	\end{equation}
	Here, the constant $ C $ above only depends on the function $ u(v) $. Thus we have 
	\begin{equation}\nonumber
	\begin{split}
	&\int_{\mathbb{Z}^{3}_{k}}\left\langle k \right\rangle^{\zeta p}(\int_{0}^{T}|(\hat{\Gamma}(\hat{f},\hat{g}),u(v))_{L^{2}_{v}}|^{2}\,d\tau)^{\frac{p}{2}}\,d\,\Sigma(k)\\
	\leq&C\{\int_{\mathbb{Z}^{3}_{k}}[\int_{\mathbb{Z}^{3}_{l}}\left\langle k \right\rangle^{\zeta}(\sup_{\tau\in[0,T]}\|\hat{f}(\tau,k-l)\|_{L^{2}_{v}})(\int_{0}^{T}|\hat{g}(\tau,l)|_{D}^{2}\,d\tau)^{\frac{1}{2}}\,d\,\Sigma(l)]^{p}\,d\,\Sigma(k)\}^{\frac{1}{p}\cdot p}\\
	&+C\{\int_{\mathbb{Z}^{3}_{k}}[\int_{\mathbb{Z}^{3}_{l}}\left\langle k \right\rangle^{\zeta}(\sup_{\tau\in[0,T]}\|\hat{g}(\tau,l)\|_{L^{2}_{v}})(\int_{0}^{T}|\hat{f}(\tau,k-l)|_{D}^{2}\,d\tau)^{\frac{1}{2}}\,d\,\Sigma(l)]^{p}\,d\,\Sigma(k)\}^{\frac{1}{p}\cdot p}.
	\end{split}
	\end{equation}
	For the former term above, we invoke the inequality $ \langle k \rangle^{h}\leq C(\langle k-l \rangle^{h}+\langle l \rangle^{h}) $ and Minkowski's Inequality again to deduce that
	\begin{equation}\nonumber
	\begin{split}
	&\{\int_{\mathbb{Z}^{3}_{k}}[\int_{\mathbb{Z}^{3}_{l}}\left\langle k \right\rangle^{\zeta}(\sup_{\tau\in[0,T]}\|\hat{f}(\tau,k-l)\|_{L^{2}_{v}})(\int_{0}^{T}|\hat{g}(\tau,l)|_{D}^{2}\,d\tau)^{\frac{1}{2}}\,d\,\Sigma(l)]^{p}\,d\,\Sigma(k)\}^{\frac{1}{p}\cdot p}\\
	&\leq C\{\int_{\mathbb{Z}^{3}_{k}}[\int_{\mathbb{Z}^{3}_{l}}\left\langle k-l \right\rangle^{\zeta}(\sup_{\tau\in[0,T]}\|\hat{f}(\tau,k-l)\|_{L^{2}_{v}})(\int_{0}^{T}|\hat{g}(\tau,l)|_{D}^{2}\,d\tau)^{\frac{1}{2}}\,d\,\Sigma(l)]^{p}\,d\,\Sigma(k)\}^{\frac{1}{p}\cdot p}\\
	&+C\{\int_{\mathbb{Z}^{3}_{k}}[\int_{\mathbb{Z}^{3}_{l}}\left\langle l \right\rangle^{\zeta}(\sup_{\tau\in[0,T]}\|\hat{f}(\tau,k-l)\|_{L^{2}_{v}})(\int_{0}^{T}|\hat{g}(\tau,l)|_{D}^{2}\,d\tau)^{\frac{1}{2}}\,d\,\Sigma(l)]^{p}\,d\,\Sigma(k)\}^{\frac{1}{p}\cdot p}\\
	&\leq C\{\int_{\mathbb{Z}^{3}_{l}}[\int_{\mathbb{Z}^{3}_{k}}\left\langle k-l \right\rangle^{\zeta p}(\sup_{\tau\in[0,T]}\|\hat{f}(\tau,k-l)\|_{L^{2}_{v}})^{p}(\int_{0}^{T}|\hat{g}(\tau,l)|_{D}^{2}\,d\tau)^{\frac{p}{2}}
	\,d\,\Sigma(k)]^{\frac{1}{p}}\,d\,\Sigma(l)\}^{p}\\
	&+C\{\int_{\mathbb{Z}^{3}_{l}}[\int_{\mathbb{Z}^{3}_{k}}(\sup_{\tau\in[0,T]}\|\hat{f}(\tau,l)\|_{L^{2}_{v}})^{p}(\int_{0}^{T}|\left\langle k-l \right\rangle^{\zeta}\hat{g}(\tau,k-l)|_{D}^{2}\,d\tau)^{\frac{p}{2}}\,d\,\Sigma(k)]^{\frac{1}{p}}\,d\,\Sigma(l)\}^{p}\\
	&=C\|f\|_{W_{k}^{\zeta,p}L^{\infty}_{T}L^{2}_{v}}^{p}\|(\int_{0}^{T}|\hat{g}(\tau,l)|_{D}^{2}\,d\tau)^{\frac{1}{2}}\|_{l^{1}_{l}}^{p}+C\|\sup_{\tau\in[0,T]}\|\hat{f}(\tau,l)\|_{L^{2}_{v}}\|^{p}_{l^{1}_{l}}\|g\|_{W_k^{\zeta,p}L^2_TL^2_v,D}^{p}\\
	&\leq C\|f\|_{W_{k}^{\zeta,p}L^{\infty}_{T}L^{2}_{v}}^{p}\|g\|_{W_k^{\zeta,p}L^2_TL^2_v,D}^{p}.
	\end{split}
	\end{equation}
	Here, we applied the inequalities
	\begin{gather*}
	\|(\int_{0}^{T}|\hat{g}(\tau,l)|_{D}^{2}\,d\tau)^{\frac{1}{2}}\|_{l^{1}_{l}}\leq \|g\|_{W_k^{\zeta,p}L^2_TL^2_v,D}
	\end{gather*}
	and
	\begin{gather*}
	\|\sup_{\tau\in[0,T]}\|\hat{f}(\tau,l)\|_{L^{2}_{v}}\|_{l^{1}_{l}}\leq \|f\|_{W_{k}^{\zeta,p}L^{\infty}_{T}L^{2}_{v}},
	\end{gather*}
	which have been proved in Lemma \ref{lemma1}. Similarly, for the latter one, we can deduce that
	\begin{equation}\nonumber
	\begin{split}
	&\{\int_{\mathbb{Z}^{3}_{k}}[\int_{\mathbb{Z}^{3}_{l}}\left\langle k \right\rangle^{\zeta}(\sup_{\tau\in[0,T]}\|\hat{g}(\tau,l)\|_{L^{2}_{v}})(\int_{0}^{T}|\hat{f}(\tau,k-l)|_{D}^{2}\,d\tau)^{\frac{1}{2}}\,d\,\Sigma(l)]^{p}\,d\,\Sigma(k)\}^{\frac{1}{p}\cdot p}\\
	&\leq C\|g\|_{W_{k}^{\zeta,p}L^{\infty}_{T}L^{2}_{v}}^{p}\|f\|_{W_k^{\zeta,p}L^2_TL^2_v,D}^{p}.
	\end{split}
	\end{equation}
	Thus we have
	\begin{equation}\nonumber
	\begin{split}
	&\int_{\mathbb{Z}^{3}_{k}}\left\langle k \right\rangle^{\zeta p}(\int_{0}^{T}|(\hat{\Gamma}(\hat{f},\hat{g}),u(v))_{L^{2}_{v}}|^{2}\,d\tau)^{\frac{p}{2}}\,d\,\Sigma(k)\\
	&\lesssim\|f\|_{W^{\zeta,p}_{k}L^{\infty}_{T}L^{2}_{v}}^{p}\|g\|_{W^{\zeta,p}_{k}L^{2}_{T}L^{2}_{v,D}}^{p}+\|g\|_{W^{\zeta,p}_{k}L^{\infty}_{T}L^{2}_{v}}^{p}\|f\|_{W^{\zeta,p}_{k}L^{2}_{T}L^{2}_{v,D}}^{p}.
	\end{split}
	\end{equation}
	Notice that there exists a constant $ C>0 $ such that it holds for any $ D,E\geq0 $ that
	\begin{equation}\nonumber
	\frac{1}{C}(D+E)\leq(D^{p}+E^{p})^{\frac{1}{p}}\leq C(D+E).
	\end{equation}
	Applying the above inequality, we finally deduce that
	\begin{equation}\nonumber
	\begin{split}
	&(\int_{\mathbb{Z}^{3}_{k}}\left\langle k \right\rangle^{\zeta p}(\int_{0}^{T}|(\hat{\Gamma}(\hat{f},\hat{g}),u(v))_{L^{2}_{v}}|^{2}\,d\tau)^{\frac{p}{2}}\,d\,\Sigma(k))^{\frac{1}{p}}\\
	&\lesssim\|f\|_{W^{\zeta,p}_{k}L^{\infty}_{T}L^{2}_{v}}\|g\|_{W^{\zeta,p}_{k}L^{2}_{T}L^{2}_{v,D}}+\|g\|_{W^{\zeta,p}_{k}L^{\infty}_{T}L^{2}_{v}}\|f\|_{W^{\zeta,p}_{k}L^{2}_{T}L^{2}_{v,D}}.
	\end{split} 
	\end{equation}
This proves \eqref{lemma3.p1} and completes the proof of Lemma \ref{lemma3}.
\end{proof}

Now we are able to prove the main estimate for the hard potential case. 
\begin{lemma}\label{lemma6}
	Let $ \gamma+2s\geq0 $. Assume all the assumptions of Theorem {\rm\ref{theorem1}} hold true. It holds that
	\begin{equation}\label{lemma6.p1}
	\begin{split}
	&\|f\|_{W^{\zeta,p}_{k}L^{\infty}_{T}L^{2}_{v}}+\|f\|_{W^{\zeta,p}_{k}L^{2}_{T}L^{2}_{v,D}}\\
	&\lesssim\|f_{0}\|_{W^{\zeta,p}_{k}L^{2}_{v}}+\|f\|_{W^{\zeta,p}_{k}L^{\infty}_{T}L^{2}_{v}}\|f\|_{W^{\zeta,p}_{k}L^{2}_{T}L^{2}_{v,D}}.
	\end{split}
	\end{equation}
\end{lemma}
\begin{proof}
	Notice that there exists $ C>0 $ such that
	\begin{equation}\nonumber
    \frac{1}{C}\|\FP f\|_{W^{\zeta,p}_{k}L^{2}_{T}L^{2}_{v,D}}\leq \|[a,b,c]\|_{W^{\zeta,p}_{k}L^{2}_{T}}\leq C\|\FP f\|_{W^{\zeta,p}_{k}L^{2}_{T}L^{2}_{v,D}}.
	\end{equation}
	The desired estimate \eqref{lemma6.p1} follows directly from Lemma \ref{lemma1}, Lemma \ref{lemma2}, Lemma \ref{lemma3} and the above inequality.
\end{proof}
Now we have finished the estimates needed for the hard potential case. To treat the soft potnetial case, especially the long-time decay rate in that case, the velocity weight defined as (\ref{dvw}) is necessary. We would like to prove an estimate that is similar to Lemma \ref{lemma6}. We first prove a lemma.

\begin{lemma}\label{lemma5}
	Let $ \gamma+2s<0 $. If $ (q,\theta) $ satisfies {\rm(\ref{vw})} then we have
	\begin{equation}\label{lemma5.p1}
	\begin{split}
	&\{\int_{\Z^{3}_{k}}\left\langle k \right\rangle^{\zeta p}(\int_{0}^{T}|(\hat{\Gamma}(\hat{f},\hat{g}),w^{2}_{q,\theta}\hat{h})_{L^{2}_{v}}|\,d\tau)^{\frac{p}{2}}\,d\,\Sigma(k)\}^{\frac{1}{p}}\\
	\leq &C(\|w_{q,\theta}f\|_{W^{\zeta,p}_{k}L^{\infty}_{T}L^{2}_{v}}\|w_{q,\theta}g\|_{W^{\zeta,p}_{k}L^{2}_{T}L^{2}_{v,D}}+\|w_{q,\theta}g\|_{W^{\zeta,p}_{k}L^{\infty}_{T}L^{2}_{v}}\|w_{q,\theta}f\|_{W^{\zeta,p}_{k}L^{2}_{T}L^{2}_{v,D}})\\
	&+\eta\|w_{q,\theta}h\|_{W^{\zeta,p}_{k}L^{2}_{T}L^{2}_{v}},
	\end{split}
	\end{equation}
	where the constant $ C $ only depends on $ \eta $.
\end{lemma}
\begin{proof}
	From the proof of \cite[Lemma 4.2]{DLSS}, we know that
	\begin{equation}\nonumber
	\begin{split}
	&|(\hat{\Gamma}(\hat{f},\hat{g}),w^{2}_{q,\theta}\hat{h})_{L^{2}_{v}}|\\
	\lesssim& \int_{\Z^{3}_{l}}(\|w_{q,\theta}\hat{f}(k-l)\|_{L^{2}_{v}}|w_{q,\theta}\hat{g}(l)|_{D}+|w_{q,\theta}\hat{f}(k-l)|_{D}\|w_{q,\theta}\hat{g}(l)\|_{L^{2}_{v}})\\
	&\times|w_{q,\theta}\hat{h}(k)|_{D}\,d\,\Sigma(l).
	\end{split}
	\end{equation}
	Then we have 
	\begin{equation}\nonumber
	\begin{split}
	&\int_{\mathbb{Z}^{3}_{k}}\left\langle k \right\rangle^{\zeta p}(\int_{0}^{T}|(\hat{\Gamma}(\hat{f},\hat{g}),w^{2}_{q,\theta}\hat{h})_{L^{2}_{v}}|\,d\tau)^{\frac{p}{2}}\,d\,\Sigma(k)\\
	&\lesssim\int_{\mathbb{Z}^{3}_{k}}\left\langle k \right\rangle^{\zeta p}\{\int_{0}^{T}(|w_{q,\theta}\hat{h}(k)|_{D})(\int_{\mathbb{Z}^{3}_{l}}\|w_{q,\theta}\hat{f}(k-l)\|_{L^{2}_{v}}|w_{q,\theta}\hat{g}(l)|_{D}\,d\,\Sigma(l))\,d\tau\}^{\frac{p}{2}}\,d\,\Sigma(k)\\
	&+\int_{\mathbb{Z}^{3}_{k}}\left\langle k \right\rangle^{\zeta p}\{\int_{0}^{T}(|w_{q,\theta}\hat{h}(k)|_{D})(\int_{\mathbb{Z}^{3}_{l}}\|w_{q,\theta}\hat{g}(l)\|_{L^{2}_{v}}|w_{q,\theta}\hat{f}(k-l)|_{D}\,d\,\Sigma(l))\,d\tau\}^{\frac{p}{2}}\,d\,\Sigma(k).
	\end{split}
	\end{equation}
	Applying the same method used to estimate (\ref{ineq}), we can deduce that
	\begin{equation}\nonumber
	\begin{split}
	&\int_{\mathbb{Z}^{3}_{k}}\left\langle k \right\rangle^{\zeta p}\{\int_{0}^{T}(|w_{q,\theta}\hat{h}(k)|_{D})(\int_{\mathbb{Z}^{3}_{l}}\|w_{q,\theta}\hat{f}(k-l)\|_{L^{2}_{v}}|w_{q,\theta}\hat{g}(l)|_{D}\,d\,\Sigma(l))\,d\tau\}^{\frac{p}{2}}\,d\,\Sigma(k)\\
	&\leq \eta\|w_{q,\theta}h\|_{W^{\zeta,p}_{k}L^{2}_{T}L^{2}_{v}}^{p}+C_{\eta}\|w_{q,\theta}f\|_{W^{\zeta,p}_{k}L^{\infty}_{T}L^{2}_{v}}^{p}\|w_{q,\theta}g\|_{W^{\zeta,p}_{k}L^{2}_{T}L^{2}_{v,D}}^{p};\\
	&\int_{\mathbb{Z}^{3}_{k}}\left\langle k \right\rangle^{\zeta p}\{\int_{0}^{T}(|w_{q,\theta}\hat{h}(k)|_{D})(\int_{\mathbb{Z}^{3}_{l}}\|w_{q,\theta}\hat{g}(l)\|_{L^{2}_{v}}|w_{q,\theta}\hat{f}(k-l)|_{D}\,d\,\Sigma(l))\,d\tau\}^{\frac{p}{2}}\,d\,\Sigma(k)\\
	&\leq\eta\|w_{q,\theta}h\|_{W^{\zeta,p}_{k}L^{2}_{T}L^{2}_{v}}^{p}+C_{\eta}\|w_{q,\theta}g\|_{W^{\zeta,p}_{k}L^{\infty}_{T}L^{2}_{v}}^{p}\|w_{q,\theta}f\|_{W^{\zeta,p}_{k}L^{2}_{T}L^{2}_{v,D}}^{p}.
	\end{split}
	\end{equation}
	Combining all above, we prove \eqref{lemma5.p1} and finish the proof of Lemma \ref{lemma5}.
\end{proof}

Applying Lemma \ref{lemma5}, we are now able to deduce a similar estimate as Lemma \ref{lemma6} for the soft potential case including the velocity weight defined as (\ref{dvw}).

\begin{lemma}\label{lemma7}
	Let $ \gamma+2s<0 $. If $ (q,\theta) $ satisfies {\rm(\ref{vw})}, then we have
	\begin{equation}\label{lemma7.p1}
	\begin{split}
	&\|w_{q,\theta}f\|_{W^{\zeta,p}_{k}L^{\infty}_{T}L^{2}_{v}}+\|w_{q,\theta}f\|_{W^{\zeta,p}_{k}L^{2}_{T}L^{2}_{v,D}}\\
	&\lesssim\|w_{q,\theta}f_{0}\|_{W^{\zeta,p}_{k}L^{2}_{v}}+\|w_{q,\theta}f\|_{W^{\zeta,p}_{k}L^{\infty}_{T}L^{2}_{v}}\|w_{q,\theta}f\|_{W^{\zeta,p}_{k}L^{2}_{T}L^{2}_{v,D}}.
	\end{split}
	\end{equation}
\end{lemma}
\begin{proof}
	Taking Fourier Transform with respect to $ x $ variable to the equation
	\begin{equation}\nonumber
	\pa_tf+v\cdot\na_xf
	+Lf=\Ga(f,f),
	\end{equation}
	we have 
	\begin{equation}\nonumber
	\pa_t\hat{f}(t,k,v)+iv\cdot\hat{f}(t,k,v)
	+L\hat{f}(t,k,v)=\hat{\Ga}(\hat{f},\hat{f})(t,k,v).
	\end{equation}
	Taking product with $ w_{q,\theta}^{2}\bar{\hat{f}} $ and taking real part, we have
	\begin{equation}\nonumber
	\frac{1}{2}\frac{d}{dt}|w_{q,\theta}\hat{f}|^{2}+\Re(L\hat{f},w_{q,\theta}^{2}\hat{f})=\Re(\Ga(\hat{f},\hat{f}),w_{q,\theta}^{2}\hat{f}).
	\end{equation}
	Here $ (\cdot,\cdot) $ denotes the inner product in the complex plane. We then integrate with respect to $ v $ and then $ t $ to deduce 
	\begin{equation}\nonumber
	\frac{1}{2}\|w_{q,\theta}\hat{f}\|^{2}_{L^{2}_{v}}+\int_{0}^{t}\Re(L\hat{f},w_{q,\theta}^{2}\hat{f})_{L^{2}_{v}}\,d\tau=\frac{1}{2}\|w_{q,\theta}\hat{f_{0}}\|^{2}_{L^{2}_{v}}+\int_{0}^{t}\Re(\Ga(\hat{f},\hat{f}),w_{q,\theta}^{2}\hat{f})_{L^{2}_{v}}\,d\tau.
	\end{equation}
	From \cite[Lemma A.3]{DLSS}, we know that
	\begin{equation}\nonumber
	(Lg,w_{q,\theta}^{2}g)_{L^{2}_{v}}\geq\delta_{q}|w_{q,\theta}g|_{D}^{2}-C|g|^{2}_{L^{2}(B_{R})}.
	\end{equation}
	Further from \cite{GS}, we know that $ |f|^{2}_{L^{2}_{\gamma+2s}}\leq|f|^{2}_{N^{s,\gamma}} $ and norm $ |f|_{N^{s,\gamma}} $ is equivalent to $ |f|_{D} $. Thus $ |f|_{L^{2}_{\gamma+2s}}\leq C|f|_{D} $. Hence we have 
	\begin{equation}\nonumber
	\begin{split}
	|g|^{2}_{L^{2}(B_{R})}=\int_{B_{R}}|g|^{2}\,dv&\leq C\int_{B_{R}}\langle v \rangle^{\gamma+2s}|g|^{2}\,dv\\
	&\leq C\int_{\R^{3}_{v}}\langle v \rangle^{\gamma+2s}|g|^{2}\,dv=C|g|^{2}_{L^{2}_{\gamma+2s}}\\&\leq C|g|_{D}^{2}.
	\end{split}
	\end{equation}
	So we have 
	\begin{equation}\nonumber
	(Lg,w_{q,\theta}^{2}g)_{L^{2}_{v}}\geq\delta_{q}|w_{q,\theta}g|_{D}^{2}-C|g|^{2}_{D}.
	\end{equation}
	Applying the above inequality, we have
	\begin{equation}\nonumber
	\begin{split}
	&\frac{1}{2}\|w_{q,\theta}\hat{f}\|^{2}_{L^{2}_{v}}+\delta_{0}\int_{0}^{t}|w_{q,\theta}\hat{f}|_{D}^{2}\,d\tau\\
	& \leq \frac{1}{2}\|w_{q,\theta}\hat{f_{0}}\|^{2}_{L^{2}_{v}}+\int_{0}^{t}\Re(\Ga(\hat{f},\hat{f}),w_{q,\theta}^{2}\hat{f})_{L^{2}_{v}}\,d\tau+C\int_{0}^{t}|\hat{f}|_{D}^{2}\,d\tau.
	\end{split}
	\end{equation} 
	As in the proof of Lemma \ref{lemma1}, we can deduce that
	\begin{equation}\nonumber
	\begin{split}
	&\|w_{q,\theta} f\|_{W^{\zeta,p}_{k}L^{\infty}_{T}L^{2}_{v}}+\{\int_{\mathbb{Z}^{3}_{k}}(\int_{0}^{T}|\left\langle k \right\rangle^{\zeta}w_{q,\theta}\hat{f}|_{D}^{2}\,d\tau)^{\frac{p}{2}}\,d\,\Sigma(k)\}^{\frac{1}{p}}\\
	\leq& C\|w_{q,\theta} f_{0}\|_{W^{\zeta,p}_{k}L^{2}_{v}}+C\|f\|_{W^{\zeta,p}_{k}L^{2}_{T}L^{2}_{v,D}}\\
	&+C[\int_{\mathbb{Z}^{3}_{k}}\left\langle k \right\rangle^{\zeta p}(\int_{0}^{T}|\Re(\hat{\Gamma}(\hat{f},\hat{f}),w_{q,\theta}^{2}\hat{f})_{L^{2}_{v}}|\,d\tau)^{\frac{p}{2}}\,d\,\Sigma(k)]^{\frac{1}{p}}.
	\end{split}
	\end{equation}
	Then we apply Lemma \ref{lemma5} to deduce
	\begin{equation}\nonumber
	\begin{split}
	&\|w_{q,\theta} f\|_{W^{\zeta,p}_{k}L^{\infty}_{T}L^{2}_{v}}+\|w_{q,\theta} f\|_{W^{\zeta,p}_{k}L^{2}_{T}L^{2}_{v,D}}\\
	\leq &C\|w_{q,\theta} f_{0}\|_{W^{\zeta,p}_{k}L^{2}_{v}}+C\eta\|w_{q,\theta} f\|_{W^{\zeta,p}_{k}L^{2}_{T}L^{2}_{v}}+C\|f\|_{W^{\zeta,p}_{k}L^{2}_{T}L^{2}_{v,D}}\\
	&+C\|w_{q,\theta} f\|_{W^{\zeta,p}_{k}L^{\infty}_{T}L^{2}_{v}}\|w_{q,\theta} f\|_{W^{\zeta,p}_{k}L^{2}_{T}L^{2}_{v,D}}.
	\end{split}
	\end{equation}
	Taking $ \eta $ to be small enough, we have
	\begin{equation}
	\begin{split}
	&\|w_{q,\theta} f\|_{W^{\zeta,p}_{k}L^{\infty}_{T}L^{2}_{v}}+\|w_{q,\theta} f\|_{W^{\zeta,p}_{k}L^{2}_{T}L^{2}_{v,D}}\\
	&\label{1}\leq C\{\|w_{q,\theta} f_{0}\|_{W^{\zeta,p}_{k}L^{2}_{v}}+\|w_{q,\theta}f\|_{W^{\zeta,p}_{k}L^{\infty}_{T}L^{2}_{v}}\|w_{q,\theta} f\|_{W^{\zeta,p}_{k}L^{2}_{T}L^{2}_{v,D}}+\|f\|_{W^{\zeta,p}_{k}L^{2}_{T}L^{2}_{v,D}}\}.
	\end{split}
	\end{equation}
	Noticing that from Lemma \ref{lemma6}, we have
	\begin{equation}\nonumber
	\begin{split}
	&\|f\|_{W^{\zeta,p}_{k}L^{2}_{T}L^{2}_{v,D}}\\
	&\leq\|f\|_{W^{\zeta,p}_{k}L^{\infty}_{T}L^{2}_{v}}+\|f\|_{W^{\zeta,p}_{k}L^{2}_{T}L^{2}_{v,D}}\\
	&\lesssim\|f_{0}\|_{W^{\zeta,p}_{k}L^{2}_{v}}+\|f\|_{W^{\zeta,p}_{k}L^{\infty}_{T}L^{2}_{v}}\|f\|_{W^{\zeta,p}_{k}L^{2}_{T}L^{2}_{v,D}}\\
	&\lesssim\|w_{q,\theta} f_{0}\|_{W^{\zeta,p}_{k}L^{2}_{v}}+\|w_{q,\theta} f\|_{W^{\zeta,p}_{k}L^{\infty}_{T}L^{2}_{v}}\|w_{q,\theta} f\|_{W^{\zeta,p}_{k}L^{2}_{T}L^{2}_{v,D}}.
	\end{split}
	\end{equation}
Applying the above inequality to bound the last term of (\ref{1}), we prove the desired estimate \eqref{lemma7.p1} and then finish the proof of Lemma \ref{lemma7}.
\end{proof}

We have concluded all the a priori estimates for the soft potential case. Now we present the following local-in-time existence result without any proof; the full details of the proof can be carried out as in \cite{DLSS,MS}.

\begin{lemma}[Local-in-time existence]\label{lemma4}
	Let all the assumptions in Theorem {\rm\ref{theorem1}} hold. Then there are constants $ \epsilon_{1}>0 $, $ T_{1}>0 $ and $ C_{1}>0 $ such that if initial data satisfy that $F_{0}(x,v)=\mu+\mu^{\frac{1}{2}}f_{0}(x,v)\geq0 $ and 
	\begin{equation*}
	\|w_{q,\theta}f_{0}\|_{W^{\zeta,p}_{k}L^{2}_{v}}\leq \epsilon_{0},
	\end{equation*}
	then the Cauchy problem {\rm(\ref{LBeq})} and {\rm(\ref{idf})} with {\rm(\ref{pt.id.cl.1})}, {\rm(\ref{pt.id.cl.2})} and {\rm(\ref{pt.id.cl.3})} for the Boltzmann equation in a torus domain admits a unique solution 
\begin{equation*}
	f\in W^{\zeta,p}_{k}L^{\infty}_{T_{1}}L^{2}_{v}\cap W^{\zeta,p}_{k}L^{2}_{T_{1}}L^{2}_{v,D}
	\end{equation*}
	satisfying	
%
	\begin{equation*}
	F(t,x,v)=\mu+\mu^{\frac{1}{2}}f(t,x,v)\geq0
	\end{equation*}
	and
	\begin{equation*}
	\|w_{q,\theta}f\|_{W^{\zeta,p}_{k}L^{\infty}_{T_{0}}L^{2}_{v}}+\|w_{q,\theta}f\|_{W^{\zeta,p}_{k}L^{2}_{T_{0}}L^{2}_{v,D}}\leq C_1\|w_{q,\theta}f_{0}\|_{W^{\zeta,p}_{k}L^{2}_{v}}.
	\end{equation*}
\end{lemma}

\section{Proof of the main theorem}
In this section, we will present the proof of Theorem \ref{theorem1}. First, we prove the global-in-time existence result along with the uniform estimate (\ref{ne}). The main tool is Lemma \ref{lemma6} and Lemma \ref{lemma7}. Then we will deal with the long-time behavior separately for the hard potnetial case and the soft potential case. The key point is to consider $ \hat{h}=e^{\lambda t^{r}}\hat{f} $, which satisfies
\begin{equation}\nonumber
\partial_{t}\hat{h}+ik\cdot v\hat{h}+L\hat{h}=e^{-\lambda t^{r}}\hat{\Gamma}(\hat{h},\hat{h})+\lambda rt^{r-1}\hat{h}.
\end{equation}
We use the same method used in this paper before to deduce a uniform estimate for $ h $. Then the anticipated decay rate of $ f $ follows. Now we present the proof.
\begin{proof}[Proof of the main theorem]
	First from Lemma \ref{lemma6} and Lemma \ref{lemma7}, we deduce that
	\begin{equation}\nonumber
	\begin{split}
	&\|w_{q,\theta}f\|_{W^{\zeta,p}_{k}L^{\infty}_{T}L^{2}_{v}}+\|w_{q,\theta}f\|_{W^{\zeta,p}_{k}L^{2}_{T}L^{2}_{v,D}}\\
	&\lesssim\|w_{q,\theta}f_{0}\|_{W^{\zeta,p}_{k}L^{2}_{v}}+\|w_{q,\theta}f\|_{W^{\zeta,p}_{k}L^{\infty}_{T}L^{2}_{v}}\|w_{q,\theta}f\|_{W^{\zeta,p}_{k}L^{2}_{T}L^{2}_{v,D}}
	\end{split}
	\end{equation}
	holds for both hard potential case and soft potential case. Then with the smallness assumption of $ \|w_{q,\theta}f_{0}\|_{W^{\zeta,p}_{k}L^{2}_{v}} $, we can deduce that
	\begin{equation}\nonumber
	\|w_{q,\theta}f\|_{W^{\zeta,p}_{k}L^{\infty}_{T}L^{2}_{v}}+\|w_{q,\theta}f\|_{W^{\zeta,p}_{k}L^{2}_{T}L^{2}_{v,D}}\lesssim\|w_{q,\theta}f_{0}\|_{W^{\zeta,p}_{k}L^{2}_{v}}.
	\end{equation}
	Thus with the local-in-time existence result Lemma \ref{lemma4}, we can get the global-in-time existence and uniqueness with 
	\begin{equation}\nonumber
	\|w_{q,\theta}f\|_{W^{\zeta,p}_{k}L^{\infty}_{T}L^{2}_{v}}+\|w_{q,\theta}f\|_{W^{\zeta,p}_{k}L^{2}_{T}L^{2}_{v,D}}\lesssim\|w_{q,\theta}f_{0}\|_{W^{\zeta,p}_{k}L^{2}_{v}}.
	\end{equation}
    	
	Next we prove the long-time decay rate. Let $ \hat{h}=e^{\lambda t^{r}}\hat{f} $, where $ \lambda>0 $ and $ r\in(0,1] $ are undetermined constants. Plug it into (\ref{feq}) to deduce that $ \hat{h} $ satisfies 
	\begin{equation}\nonumber
	\partial_{t}\hat{h}+ik\cdot v\hat{h}+L\hat{h}=e^{-\lambda t^{r}}\hat{\Gamma}(\hat{h},\hat{h})+\lambda rt^{r-1}\hat{h}
	\end{equation}
	with $ \hat{h}(0,k,v)=\hat{h}_{0}(k,v) $.
	Taking product with $ \bar{\hat{h}} $ and use the same method as in Lemma \ref{lemma1}, Lemma \ref{lemma2}, Lemma \ref{lemma3} and Lemma \ref{lemma6} to get
	\begin{equation}\nonumber
	\begin{split}
	&\|h\|_{W^{\zeta,p}_{k}L^{\infty}_{T}L^{2}_{v}}+\|h\|_{W^{\zeta,p}_{k}L^{2}_{T}L^{2}_{v,D}}\\
	&\lesssim\|h_{0}\|_{W^{\zeta,p}_{k}L^{2}_{v}}+(\lambda r)^{\frac{1}{2}}\{\int_{\Z^{3}_{k}}(\int_{0}^{T}\tau^{r-1}\|\langle k\rangle^{\zeta}\hat{h}(\tau,k)\|^{2}_{L^{2}_{v}}\,d\tau)^{\frac{p}{2}}\,d\,\Sigma(k)\}^{\frac{1}{p}}.
	\end{split}
	\end{equation}
	
	For the hard potential case, set $ r=1 $. We have
	\begin{equation}\nonumber
	\begin{split}
	&(\lambda )^{\frac{1}{2}}\{\int_{\Z^{3}_{k}}(\int_{0}^{T}\|\langle k\rangle^{\zeta}\hat{h}(\tau,k)\|^{2}_{L^{2}_{v}}\,d\tau)^{\frac{p}{2}}\,d\,\Sigma(k)\}^{\frac{1}{p}}\\
	&\leq(\lambda )^{\frac{1}{2}}\{\int_{\Z^{3}_{k}}(\int_{0}^{T}\|\langle v \rangle^{\gamma+2s}\langle k\rangle^{\zeta}\hat{h}(\tau,k)\|^{2}_{L^{2}_{v}}\,d\tau)^{\frac{p}{2}}\,d\,\Sigma(k)\}^{\frac{1}{p}}\\
	& =(\lambda )^{\frac{1}{2}}\{\int_{\Z^{3}_{k}}(\int_{0}^{T}\|\langle k\rangle^{\zeta}\hat{h}\|_{L_{\gamma+2s}}^{2}\,d\tau)^{\frac{p}{2}}\,d\,\Sigma(k)\}^{\frac{1}{p}}\\
	&\leq(\lambda )^{\frac{1}{2}}\{\int_{\Z^{3}_{k}}(\int_{0}^{T}|\langle k\rangle^{\zeta}\hat{h}|_{D}^{2}\,d\tau)^{\frac{p}{2}}\,d\,\Sigma(k)\}^{\frac{1}{p}}=(\lambda )^{\frac{1}{2}}\|h\|_{W^{\zeta,p}_{k}L^{2}_{T}L^{2}_{v,D}}.
	\end{split}
	\end{equation}
	Taking $ \lambda $ to be small enough, we deduce 
	\begin{equation}\nonumber
	\|h\|_{W^{\zeta,p}_{k}L^{\infty}_{T}L^{2}_{v}}+\|h\|_{W^{\zeta,p}_{k}L^{2}_{T}L^{2}_{v,D}}\lesssim\|h_{0}\|_{W^{\zeta,p}_{k}L^{2}_{v}}.
	\end{equation}
	This is the estimate needed for the decay rate in the hard potential case. Thus we have 
	\begin{equation}\nonumber
	e^{\lambda t^{r}}\|f\|_{W^{\zeta,p}_{k}L^{2}_{v}}=\|h\|_{W^{\zeta,p}_{k}L^{2}_{v}}\lesssim \|f_{0}\|_{W^{\zeta,p}_{k}L^{2}_{v}}.
	\end{equation}
	Finally we deduce
	\begin{equation}\nonumber
	\|f\|_{W^{\zeta,p}_{k}L^{2}_{v}}\lesssim e^{-\lambda t^{r}}\|f_{0}\|_{W^{\zeta,p}_{k}L^{2}_{v}}.
	\end{equation}
	Taking $ \kappa=r=1 $ to finish the proof for the hard potential case.
	
	For the soft potential case, we set two undetermined constants $ \rho>0 $ and $ r'>0 $ and let
	\begin{equation}\nonumber
	\FE=\{\langle v \rangle\leq\rho \tau^{r'}\}.
	\end{equation}
	Then we have
	\begin{equation}\nonumber
	\begin{split}
	&(\lambda r)^{\frac{1}{2}}\{\int_{\Z^{3}_{k}}(\int_{0}^{T}\tau^{r-1}\|\langle k\rangle^{\zeta}\hat{h}(\tau,k)\|^{2}_{L^{2}_{v}}\,d\tau)^{\frac{p}{2}}\,d\,\Sigma(k)\}^{\frac{1}{p}}\\
	\leq&(\lambda r)^{\frac{1}{2}}\{\int_{\Z^{3}_{k}}(\int_{0}^{T}\int_{\R^{3}_{v}}\tau^{r-1}\mathbf{1}_{\FE}|\langle k\rangle^{\zeta}\hat{h}(\tau,k)|^{2}\,dv\,d\tau)^{\frac{p}{2}}\,d\,\Sigma(k)\}^{\frac{1}{p}}\\
	&+(\lambda r)^{\frac{1}{2}}\{\int_{\Z^{3}_{k}}(\int_{0}^{T}\int_{\R^{3}_{v}}\tau^{r-1}\mathbf{1}_{\FE^{c}}|\langle k\rangle^{\zeta}\hat{h}(\tau,k)|^{2}\,dv\,d\tau)^{\frac{p}{2}}\,d\,\Sigma(k)\}^{\frac{1}{p}}.
	\end{split}
	\end{equation}
	The former one is denoted as $ I_{1} $ and the latter one is denoted as $ I_{2} $. We now let $ r $ and $ r' $ satisfy $ \frac{r-1}{r'}=\gamma+2s<0 $. Since $ \frac{r-1}{r'}<0 $, we have $ (\frac{\langle v \rangle}{\rho})^{\gamma+2s}\geq \tau^{r'(\gamma+2s)}=\tau^{r-1} $ on $ \FE $. Thus we have
	\begin{equation}\nonumber
	\begin{split}
	I_{1}&\leq (\lambda r)^{\frac{1}{2}}\rho^{{-\frac{r-1}{2r'}}}\{\int_{\Z^{3}_{k}}(\int_{0}^{T}\int_{\R^{3}_{v}}\langle v \rangle^{\gamma+2s}|\langle k\rangle^{\zeta}\hat{h}(\tau,k)|^{2}\,dv\,d\tau)^{\frac{p}{2}}\,d\,\Sigma(k)\}^{\frac{1}{p}}\\
	&=(\lambda r)^{\frac{1}{2}}\rho^{{-\frac{r-1}{2r'}}}\{\int_{\Z^{3}_{k}}(\int_{0}^{T}|\langle k\rangle^{\zeta}\hat{h}(\tau,k)|^{2}_{L^{2}_{\gamma+2s}}\,d\tau)^{\frac{p}{2}}\,d\,\Sigma(k)\}^{\frac{1}{p}}.
	\end{split}
	\end{equation}
	Since we have known that $ |\langle k\rangle^{\zeta}\hat{h}(\tau,k)|^{2}_{L^{2}_{\gamma+2s}}\leq C|\langle k\rangle^{\zeta}\hat{h}(\tau,k)|_{D}^{2} $, we have
	\begin{equation}\nonumber
	\begin{split}
	I_{1}&\leq C(\lambda r)^{\frac{1}{2}}\rho^{{-\frac{r-1}{2r'}}}\{\int_{\Z^{3}_{k}}(\int_{0}^{T}|\langle k\rangle^{\zeta}\hat{h}(\tau,k)|^{2}_{D}\,d\tau)^{\frac{p}{2}}\,d\,\Sigma(k)\}^{\frac{1}{p}}\\
	&= C(\lambda r)^{\frac{1}{2}}\rho^{{-\frac{r-1}{2r'}}}\|h\|_{W^{\zeta,p}_{k}L^{2}_{T}L^{2}_{v,D}}.
	\end{split}
	\end{equation}
	Then we estimate $ I_{2} $. Since on $ \FE^{c} $, we have $ \langle v \rangle>\rho \tau^{r'} $. So $ w_{q,\theta}=e^{q\langle v \rangle}\geq e^{q\rho\tau^{r'}} $. Thus $ 1\leq w_{q,\theta}^{2}e^{-2q\rho\tau^{r'}} $ on $ \FE^{c} $. So we can estimate:
	\begin{equation}\nonumber
	\begin{split}
	I_{2}&\leq (\lambda r)^{\frac{1}{2}}\{\int_{\Z^{3}_{k}}[\int_{0}^{T}\tau^{r-1}e^{2\lambda\tau^{r}}e^{-2q\rho\tau^{r'}}(\int_{\FE^{c}}|\langle k\rangle^{\zeta}{w_{q,\theta}}\hat{f}|^{2}\,dv)\,d\tau]^{\frac{p}{2}}\,d\,\Sigma(k)\}^{\frac{1}{p}}\\
	&\leq  (\lambda r)^{\frac{1}{2}}\{\int_{\Z^{3}_{k}}\langle k\rangle^{\zeta p}\sup_{t\in [0,T]}\|w_{q,\theta}\hat{f}\|_{L^{2}_{v}}^{p}\,d\,\Sigma(k)\}^{\frac{1}{p}}(\int_{0}^{T}\tau^{r-1}e^{2\lambda\tau^{r}}e^{-2q\rho\tau^{r'}}\,d\tau)^{\frac{1}{2}}.
	\end{split}
	\end{equation}
	Now we set $ r=r' $, which implies $ r=\frac{1}{1+|\gamma+2s|} $, and let $ \lambda $ be small enough such that $ 2\lambda<2q\rho $. Then we have
	\begin{equation}\nonumber
	\begin{split}  \int_{0}^{T}\tau^{r-1}e^{2\lambda\tau^{r}}e^{-2q\rho\tau^{r'}}\,d\tau&=\int_{0}^{T}\tau^{r-1}e^{2\lambda\tau^{r}}e^{-2q\rho\tau^{r'}}\,d\tau\\
	&\leq\int_{0}^{\infty}\tau^{r-1}e^{2(\lambda-q\rho)\tau^{r}}\,d\tau< \infty .
	\end{split}
	\end{equation}
	Thus we can deduce that 
	\begin{equation}\nonumber
	I_{2}\leq C(\lambda r)^{\frac{1}{2}}\|w_{q,\theta}f\|_{W^{\zeta,p}_{k}L^{2}_{T}L^{2}_{v,D}}.
	\end{equation}
	Combining with the existence theorm, we further have
	\begin{equation}\nonumber
	I_{2}\leq C(\lambda r)^{\frac{1}{2}}\|w_{q,\theta}f_{0}\|_{W^{\zeta,p}_{k}L^{2}_{v}}.
	\end{equation}
	Thus we have 
	\begin{equation}\nonumber
	\begin{split}
	&\|h\|_{W^{\zeta,p}_{k}L^{\infty}_{T}L^{2}_{v}}+\|h\|_{W^{\zeta,p}_{k}L^{2}_{T}L^{2}_{v,D}}\\
	&\lesssim
	\|h_{0}\|_{W^{\zeta,p}_{k}L^{2}_{v}}+C(\lambda r)^{\frac{1}{2}}\rho^{-\frac{2(r-1)}{r'}}\|h\|_{W^{\zeta,p}_{k}L^{2}_{T}L^{2}_{v,D}}+C(\lambda r)^{\frac{1}{2}}\|w_{q,\theta}f_{0}\|_{W^{\zeta,p}_{k}L^{2}_{v}}.
	\end{split}
	\end{equation}
	Taking $ \lambda $ and $ \rho $ to be small enough and using the fact $ \|h_{0}\|_{W^{\zeta,p}_{k}L^{2}_{v}}\leq \|w_{q,\theta}f_{0}\|_{W^{\zeta,p}_{k}L^{2}_{v}}  $, we can deduce that
	\begin{equation}\nonumber
	\|h\|_{W^{\zeta,p}_{k}L^{\infty}_{T}L^{2}_{v}}+\|h\|_{W^{\zeta,p}_{k}L^{2}_{T}L^{2}_{v,D}}\lesssim \|w_{q,\theta}f_{0}\|_{W^{\zeta,p}_{k}L^{2}_{v}}.
	\end{equation}
	Thus we have 
	\begin{equation}\nonumber
	e^{\lambda t^{r}}\|f\|_{W^{\zeta,p}_{k}L^{2}_{v}}=\|h\|_{W^{\zeta,p}_{k}L^{2}_{v}}\lesssim \|w_{q,\theta}f_{0}\|_{W^{\zeta,p}_{k}L^{2}_{v}}.
	\end{equation}
	Finally we deduce
	\begin{equation}\nonumber
	\|f\|_{W^{\zeta,p}_{k}L^{2}_{v}}\lesssim e^{-\lambda t^{r}}\|w_{q,\theta}f_{0}\|_{W^{\zeta,p}_{k}L^{2}_{v}}.
	\end{equation}
	We take $ \kappa=r=\frac{1}{1+|\gamma+2s|} $ to finish the proof for the soft potential case. Therefore, \eqref{thm1.decay} holds true and the proof of Theorem \ref{theorem1} is completed. 
\end{proof}

\medskip
\noindent{\bf Acknowledgments:} The author was supervised by Professor Renjun Duan from The Chinese University of Hong Kong. The author would like to express sincere gratitude to Professor Renjun Duan for his patience and guidence throughout the research project in the summer of 2020.

\end{document}